\documentclass[11pt,a4paper]{article}

\usepackage[T1]{fontenc}
\usepackage[utf8]{inputenc}
\usepackage[english]{babel}
\usepackage{amsmath,amssymb,amsthm}
\usepackage{mathtools}
\usepackage[margin=3cm]{geometry}
\usepackage{hyperref}
\hypersetup{colorlinks=true,linkcolor=blue,citecolor=blue,urlcolor=blue}

\theoremstyle{plain}
\newtheorem{theorem}{Theorem}
\newtheorem{lemma}[theorem]{Lemma}
\newtheorem{proposition}[theorem]{Proposition}
\newtheorem{corollary}[theorem]{Corollary}
\newtheorem{conjecture}[theorem]{Conjecture}

\theoremstyle{definition}
\newtheorem{example}[theorem]{Example}
\newtheorem{remark}[theorem]{Remark}
\newtheorem{question}[theorem]{Question}

\newcommand{\bfc}{\mathrm{BFC}}
\newcommand{\Aut}{\operatorname{Aut}}
\newcommand{\Z}{Z}
\newcommand{\dd}{\operatorname{d}}
\newcommand{\AGL}{\operatorname{AGL}}
\newcommand{\Ind}{\operatorname{Ind}}

\numberwithin{equation}{section}

\title{On the order of a finite group with trivial centre\\ and bounded conjugacy classes}
\author{Ilya Gorshkov}
\date{ }
\begin{document}
\maketitle

\begin{abstract}
A group is called an $n$-$\bfc$-group if each of its conjugacy classes is finite
and contains at most $n$ elements. By a theorem of B.~H.~Neumann the derived
subgroup of such a group is finite of $n$-bounded order, whereas the order of
the group itself need not be bounded, as extraspecial $p$-groups show. We prove
that an $n$-$\bfc$-group with trivial centre is finite of order at most
$n^{64(\log n)^{5}}$. Passing to the quotient by the hypercentre, we deduce that no
assumption on the structure of the group is needed at all: for an arbitrary
finite $n$-$\bfc$-group $G$ one has $|G:Z_\infty(G)|<n^{64(\log n)^{5}}$. The finiteness of the group is required here only for the correct definition of the hypercentre.

\end{abstract}

\section{Introduction}

In 1954 B.~H.~Neumann~\cite{Neumann54} introduced the class of $\bfc$-groups.
A group $G$ is called an \emph{$n$-$\bfc$-group} if each of its conjugacy
classes is finite and contains at most $n$ elements. The least such $n$ is
called the \emph{$\bfc$-number} of $G$; we denote it by $\bfc(G)$. A group is
called a $\bfc$-group if it is an $n$-$\bfc$-group for some $n$. Neumann proved
that a group is a $\bfc$-group if and only if its derived subgroup $G'$ is
finite. In particular, for an $n$-$\bfc$-group the order $|G'|$ is bounded by a
function of $n$.

Subsequent work refined the upper bound for the order of the derived subgroup of
an $n$-$\bfc$-group. The first explicit bound was obtained by
Wiegold~\cite{Wiegold57}; it was substantially improved by
Macdonald~\cite{Macdonald61} and by P.~M.~Neumann and
M.~R.~Vaughan-Lee~\cite{NVL77}.

Wiegold conjectured that in fact $|G'|\le n^{\frac12(1+\log n)}$. For nilpotent
groups this bound was proved by Vaughan-Lee~\cite{VL74}, the extremal examples
being the free groups of exponent $p$. The best bound known at present in the
general case is due to R.~Guralnick and A.~Mar\'oti~\cite{GM11}. Their method
rests on estimates for the average dimension of fixed point spaces. Note that the result of Guralnick and Mar\'oti depends on the
classification of finite simple groups (CFSG).

We shall be interested in the following question: under what conditions does the
$\bfc$-number bound the order not only of the derived subgroup, but of the whole
group?

In this direction results are known under the assumption that the group has no
nontrivial abelian normal subgroup, that is, $F(G)=1$, where $F(G)$ is the
Fitting subgroup of $G$. D. Segal and A. Shalev \cite{SS99} proved that in this case
$|G|<n^{4}$. Guralnick and Mar\'oti~\cite[Theorem 1.10]{GM11} improved this
bound to $|G|<n^{2}$, establishing the more general inequality, under the
assumption that the Fitting subgroup is finite,
$$
  |G|<n^{2}k\bigl(F(G)\bigr),
$$
where $k(X)$ denotes the number of conjugacy classes of a group $X$.

The condition $F(G)=1$ is, however, substantially stronger than triviality of
the centre: the groups $\AGL(1,q)$ of Proposition~\ref{prop:agl} have trivial
centre, whereas their Fitting subgroup is the regular normal subgroup of order
$q$, so that the bound above only gives $|G|<n^{3}$ for them. The main result of
the present paper covers such groups as well.

\begin{theorem}\label{thm:main}
Let $G$ be an $n$-$\bfc$-group with trivial centre. Then
$$
  |G|\leq n^{64(\log n)^5}.
$$
In particular, $|G|$ is bounded by a function of $n$.
\end{theorem}

To bound the order of $G$ we used the estimate of P.~M.~Neumann and
M.~R.~Vaughan-Lee~\cite{NVL77}, which does not use CFSG. Thus our proof does not
depend on the CFSG and is readily amenable to verification in the Lean
proof assistant. Replacing the Neumann--Vaughan-Lee bound in the theorem by the
bound obtained in~\cite{GM11} improves the bound on $|G|$ substantially;
it nevertheless remains exponential.

We did not attempt to obtain the best possible bound for the order of the group,
since our main aim was the following corollary.

\begin{corollary}\label{cor:finite}
Let $G$ be an $n$-$\bfc$-group with trivial centre. Then $G$ is a finite group
and the order of $G$ is bounded in terms of $n$.
\end{corollary}

For a subgroup $H\leq G$ and an element $g\in G$ we put
$\Ind(H,g)=|H:C_H(g)|$; in particular, $\Ind(G,g)=|g^{G}|$.

Note that triviality of the centre of $G$ in the hypothesis of the theorem is
essential.

\begin{example}\label{ex:extraspecial-intro}
Let $G$ be an extraspecial group of order $p^{1+2k}$. Then $|x^{G}|\le p$ for
every $x\in G$. Hence $G$ is a $p$-$\bfc$-group, while its order $p^{1+2k}$ is
arbitrary. Moreover $|G:\Z(G)|=p^{2k}$. Thus the quantity $|G:\Z(G)|$ is not
bounded in terms of $n$ either.
\end{example}

Nevertheless, the condition on the centre can be dispensed with altogether if
the centre is replaced by the hypercentre. Recall that the \emph{hypercentre}
$Z_\infty(G)$ of a finite group $G$ is the last term of the upper central series
$1=Z_0(G)\le Z_1(G)=\Z(G)\le Z_2(G)\le\dots$, that is, the term at which the
series stabilises.

\begin{corollary}\label{cor:hyper}
Let $G$ be a finite $n$-$\bfc$-group with $n\ge2$. Then
$$
  |G:Z_\infty(G)|<n^{64(\log n)^5}.
$$
\end{corollary}
Note that Corollary \ref{cor:hyper} requires the group to be finite. This condition is used only to ensure that the subgroup $Z_{\infty}(G)$ is defined correctly.
Here no assumption whatsoever is made on the structure of $G$. For the
extraspecial group of Example~\ref{ex:extraspecial-intro},
Corollary~\ref{cor:hyper} holds trivially, since that group is nilpotent and
$Z_\infty(G)=G$; for groups with trivial centre it turns into
Theorem~\ref{thm:main}.

The following conjecture was stated in~\cite{NVL77}.

\begin{conjecture}\label{NVL}
Let $G$ be a perfect and centreless finite group, and let $n$ be the maximum
size of a conjugacy class in $G$. Is it true that $|G|\leq n^2$?
\end{conjecture}

This conjecture is recorded in the Kourovka Notebook~\cite{Kour} as Problem
20.30.

It is of particular interest to find the exact bound for the order of $G$. We
state the following conjecture.

\begin{conjecture}\label{conj:sharp}
If $G$ is a finite $n$-$\bfc$-group with trivial centre, then $|G|\le n^2-n$.
\end{conjecture}

The bound $n^2-n$ is attained on the groups
$\AGL(1,q)=\mathbb{F}_q\rtimes\mathbb{F}_q^{\times}$
(Proposition~\ref{prop:agl}). A direct search through all groups of order at
most $1000$ with trivial centre produces no counterexamples
(\S\ref{sec:sharp}). Moreover, the truth of Conjecture~\ref{conj:sharp} implies
the truth of Conjecture~\ref{NVL}.

\begin{corollary}\label{prim}
If Conjecture~\ref{conj:sharp} is true, then so is Conjecture~\ref{NVL}.
\end{corollary}

The paper is organised as follows. In \S\ref{sec:prelim} we collect the
auxiliary facts used in the sequel. In \S\ref{sec:main} we prove
Theorem \ref{thm:main} and Corollary \ref{cor:hyper}. In \S\ref{sec:sharp} we
discuss the sharpness of the bound obtained: a series of groups attaining the
value $n^2-n$ is exhibited, and Corollary \ref{prim} is proved.
In \S\ref{sec:open} open questions are stated.

\section{Preliminaries}\label{sec:prelim}

The following lemma is known as Poincare's theorem. It is given in many
textbooks, see for instance \cite[\S4]{KM}.

\begin{lemma}\label{lem:poincare}
Let $A$ and $B$ be subgroups of a group $G$ with $|G:A|=k$ and $|G:B|=m$. Then
$|G:A\cap B|\leq km$.
\end{lemma}

\begin{lemma}\label{lem:index}
Let $G$ be an $n$-$\bfc$-group, $g\in G$ and $H\leq G$. Then $\Ind(H,g)\leq n$.
\end{lemma}
\begin{proof}
Elements that are conjugate in $H$ are clearly conjugate in $G$ as well, so
$g^{H}\subseteq g^{G}$. Since $\Ind(G,g)\leq n$, the assertion of the lemma
follows.
\end{proof}

We denote by $\dd(G)$ the minimal number of generators of a group $G$.

\begin{lemma}\label{lem:aut}
For every finite group $G$ the following hold:
\begin{enumerate}
	\item $\dd(G)\le\log|G|$;
	\item $|\Aut(G)|\le|G|^{\log|G|}$.
\end{enumerate}
\end{lemma}
\begin{proof}
If $G=\langle x_1,\dots,x_m\rangle$ and $G_i=\langle x_1,\dots,x_i\rangle$, then
in the chain $G_0<G_1<\dots<G_m=G$ every index is at least $2$, whence
$2^{m}\le|G|$ and $\dd(G)\le\log|G|$.

An automorphism is uniquely determined by the images of the elements of a
generating set, and therefore
$|\Aut(G)|\le|G|^{\dd(G)}\le|G|^{\log|G|}$.
\end{proof}

\begin{lemma}\label{lem:relgen}
Let $H$ be a subgroup of finite index in a group $C$. Then there exist elements
$c_1,\dots,c_s\in C$, where $s\le\log|C:H|$, such that
$C=\langle H,c_1,\dots,c_s\rangle$.
\end{lemma}
\begin{proof}
We construct a chain of subgroups $H=H_0<H_1<\dots<H_s=C$ by putting
$H_{i+1}=\langle H_i,c_{i+1}\rangle$ for an arbitrary element
$c_{i+1}\in C\setminus H_i$. The indices $|C:H_i|$ strictly decrease, so the
process terminates whenever $|C:H|$ is finite. Every inclusion is strict, hence
$|H_{i+1}:H_i|\ge2$. Consequently
$$
  |C:H|=\prod_{i=0}^{s-1}|H_{i+1}:H_i|\ \ge\ 2^{s},
$$
whence $s\le\log|C:H|$. By construction $C=\langle H,c_1,\dots,c_s\rangle$.
\end{proof}

\begin{lemma}\label{prop:gen}
For every finite $n$-$\bfc$-group $G$ one has $|G:\Z(G)|\le n^{\dd(G)}$. In
particular, if $\Z(G)=1$, then $|G|\le n^{\dd(G)}$.
\end{lemma}
\begin{proof}
If $G=\langle x_1,\dots,x_d\rangle$, then $\Z(G)=\bigcap_{i=1}^{d}C_G(x_i)$, and
it remains to apply Lemma~\ref{lem:poincare}.
\end{proof}

Thus, for groups with trivial centre the question of the boundedness of $|G|$ is
equivalent to the question of the boundedness of the number of generators
$\dd(G)$ in terms of $n$. Lemma~\ref{prop:gen} by itself gives no bound: the
estimate $\dd(G)\le\log|G|$ leads to the tautology $|G|\le n^{\log|G|}$. It
does, however, explain why the naive argument ``$\Z(G)$ is an intersection of
centralisers of index at most $n$'' does not work.

The following lemma was proved in~\cite{Grun35}. Since its text is written in
German and is hard to obtain, we include a simple proof.

\begin{lemma}[Gr\"un]\label{lem:grun}
If $G$ is perfect, then $\Z\bigl(G/\Z(G)\bigr)=1$.
\end{lemma}

\begin{proof}
Let $\Z_2(G)$ be the second centre, that is, the preimage of $\Z(G/\Z(G))$.
Then $[[\Z_2(G),G],G]=1$ and $[[G,\Z_2(G)],G]=1$, so that by the three subgroup
lemma $[[G,G],\Z_2(G)]=1$. Since $G'=G$, it follows that
$\Z_2(G)\le C_G(G)=\Z(G)$.
\end{proof}

\begin{lemma}[{\cite[Theorem 1]{NVL77}}]\label{ENVL}
Let $G$ be an $n$-$\bfc$-group. Then $|G'|\leq n^{\frac12(3+5\log n)}$.
\end{lemma}

\begin{lemma}[{\cite[Theorem 1.8]{GM11}}]\label{NM}
Let $G$ be an $n$-$\bfc$-group with $n>1$. Then $|G'|< n^{\frac12(7+\log n)}$.
\end{lemma}

\section{Proof of Theorem~\ref{thm:main}}\label{sec:main}

Let $G$ be an $n$-$\bfc$-group with trivial centre. Note that if $G'=1$, then
$G$ is abelian and $G=\Z(G)=1$; hence $G'\ne 1$.

Put $C=C_G(G')$. Note that $G/C$ embeds into the automorphism
group of $G'$. By Lemma~\ref{ENVL} the group $G'$ is finite, and therefore so is
$G/C$. Fix elements $g_1,\dots,g_r\in G$ whose images generate $G/C$. By
Lemma~\ref{lem:aut} we have $r=\dd(G/C)\le\log |G:C|$. Clearly
$G=\langle C,\ g_1,\dots,g_r\rangle$.

Put
$$
  a=\tfrac12\bigl(3+5\log n\bigr),
$$
so that $|G'|\le n^{a}$ by Lemma~\ref{ENVL}.

We split the rest of the proof into several steps.

\medskip
\noindent\textbf{Step 1.} \emph{$C$ is nilpotent of class at most $2$, and
$|G:C|\le n^{a^2\log n}$.}

Since $C\le G$, we have $C'\le G'$; since $C$ centralises $G'$, it centralises
$C'$ as well, that is, $C'\le \Z(C)$. Consequently $C$ is nilpotent of class at
most $2$.

Clearly $C$ is the kernel of the action of $G$ on $G'$, so that $G/C$ embeds
into $\Aut(G')$. By Lemma~\ref{lem:aut}.2 and the bound $|G'|\le n^{a}$,
$$
  |\Aut(G')|\le|G'|^{\log|G'|}\le\bigl(n^{a}\bigr)^{a\log n}=n^{a^{2}\log n}.
$$
Thus $|G:C|\le n^{a^{2}\log n}$.

Note that $r\le \log n^{a^2\log n}=a^2 (\log n)^2$.

\medskip
\noindent\textbf{Step 2.} \emph{Put $K=\bigcap_{i=1}^{r}C_C(g_i)$. Then the
following hold:}
\begin{enumerate}
	\item $|C:K|\le n^{a^2 (\log n)^2}$;
	\item $K$ \emph{is abelian.}
\end{enumerate}

By Lemma~\ref{lem:index} we have $|C:C_C(g_i)|\le n$ for every $i$, and by
Lemma~\ref{lem:poincare} the index of the intersection does not exceed $n^{r}$.
Since $r\le a^2 (\log n)^2$, the required bound follows.

Let us prove that $K$ is abelian. Let $c,d\in K$. Then $[c,d]\in C'$. On the
other hand, for every $i$ we have
$$
  [c,d]^{g_i}=[c^{g_i},d^{g_i}]=[c,d],
$$
since $c$ and $d$ commute with $g_i$ and by Step 1 $C'\leq Z(C)$. We have
$C_G([c,d])\geq \langle C,g_1,...,g_r\rangle=G$, and yields $[c,d]=1$.

\medskip
\noindent\textbf{Step 3.} \emph{$|\Z(C)|\le n^{r}$.}

Assume that there exists $b\in \Z(C)\setminus\{1\}$ commute with all of $g_1,\dots,g_r$. We have $C_G(b)\geq\langle C,g_1,\dots,g_r\rangle=G$.
Thus $b\in \Z(G)=1$; a contradiction. Therefore
$$
  \bigcap_{i=1}^{r}C_{\Z(C)}(g_i)=1 .
$$
By Lemma~\ref{lem:index} we have $|\Z(C):C_{\Z(C)}(g_i)|\le n$, and
Lemma~\ref{lem:poincare} gives
$$
  |\Z(C)|=\Bigl|\Z(C):\bigcap_{i=1}^{r}C_{\Z(C)}(g_i)\Bigr|\le n^{r}.
$$

\medskip
\noindent\textbf{Step 4.} \emph{Let $c_1,\dots,c_s\in C$ be elements with
$C=\langle K,c_1,\dots,c_s\rangle$, chosen so that $s\le\log|C:K|$; such a set
exists by Lemma~\ref{lem:relgen}, since the index $|C:K|$ is finite by Step 2.1.
Put $D=\bigcap_{j=1}^{s}C_K(c_j)$. Then $D\le \Z(C)$ and $|K:D|\le n^{s}$.}

Let $d\in D$. It centralises every $c_j$ by definition, and it centralises $K$,
since $K$ is abelian by Step 2.2 and $d\in K$. Hence $d$ centralises
$\langle K,c_1,\dots,c_s\rangle=C$, that is, $d\in\Z(C)$. The bound for $|K:D|$
follows from Lemmas~\ref{lem:index} and~\ref{lem:poincare}.

Note that, by Step 2.1, $s\le\log|C:K|\le a^2(\log n)^3$.

\begin{proof}[Proof of Theorem~\ref{thm:main}]
We have the chain of subgroups $G\ge C\ge K\ge D$, where $D\le\Z(C)$ by Step~5.
By Steps 1, 2, 4 and 3 respectively,
$$
  |G|=|G:C|\cdot|C:K|\cdot|K:D|\cdot|D|
  \ \le\ n^{a^{2}\log n}\cdot n^{r}\cdot n^{s}\cdot n^{r}
  \ =\ n^{\,a^{2}\log n+2r+s},
$$
where $|D|\le|\Z(C)|\le n^{r}$.

Put $\lambda=\log n$; since $n\ge2$, we have $\lambda\ge1$. By Step~1,
$r\le a^{2}\lambda^{2}$, and by Step~5, $s\le a^{2}\lambda^{3}$, so that the
exponent does not exceed
$$
  a^{2}\lambda+2a^{2}\lambda^{2}+a^{2}\lambda^{3}
  = a^{2}\lambda\,(1+\lambda)^{2}.
$$
Finally, $a=\tfrac12(3+5\lambda)\le4\lambda$ and
$(1+\lambda)^{2}\le4\lambda^{2}$ for $\lambda\ge1$, whence
$$
  a^{2}\lambda\,(1+\lambda)^{2}\ \le\ 16\lambda^{2}\cdot\lambda\cdot4\lambda^{2}
  \ =\ 64\lambda^{5},
$$
that is, $|G|\le n^{64(\log n)^{5}}$.
\end{proof}

\begin{proof}[Proof of Corollary~\ref{cor:hyper}]
Let $Z_\infty(G)=Z_c(G)$, that is, $Z_{c+1}(G)=Z_c(G)$, and put
$\bar G=G/Z_c(G)$. Then $\Z(\bar G)=Z_{c+1}(G)/Z_c(G)=1$.

Further, the conjugacy class of an element $\bar x\in\bar G$ is the image of the
class $x^{G}$ under the natural homomorphism, so that
$|\bar x^{\bar G}|\le|x^{G}|\le n$. Hence $\bar G$ is an $\bar n$-$\bfc$-group
with $\bar n\le n$.

If $\bar G=1$, there is nothing to prove. Otherwise $\bar n\ge2$: indeed, if
$\bar n=1$, then $\bar G$ would be abelian, and then $\bar G=\Z(\bar G)=1$.
Applying Theorem~\ref{thm:main} to $\bar G$, and using the fact that the
function $t\mapsto t^{64(\log t)^{5}}$ is non-decreasing for $t\ge2$, we obtain
$$
  |G:Z_\infty(G)|=|\bar G|<\bar n^{64(\log \bar n)^{5}}\le n^{64(\log n)^{5}} .
$$
\end{proof}

\section{Sharpness of the bound and its relation to
Conjectures~\ref{NVL} and~\ref{conj:sharp}}\label{sec:sharp}

\begin{proposition}\label{prop:agl}
Let $q>2$ be a prime power and let
$G=\AGL(1,q)=\mathbb{F}_q\rtimes\mathbb{F}_q^{\times}$. Then $\Z(G)=1$, the
$\bfc$-number of $G$ equals $n=q$, and $|G|=q(q-1)=n^{2}-n$.
\end{proposition}

\begin{proof}
Let $V=\mathbb{F}_q$ be the additive group and let $H=\mathbb{F}_q^{\times}$ act
on it by multiplication. The action is faithful and has no nonzero fixed
points, so that $C_G(V)=V$ and $\Z(G)\le V$, while $C_V(H)=0$ gives $\Z(G)=1$.
For $0\ne v\in V$ we have $C_G(v)=V$ and $|v^{G}|=q-1$. For $1\ne h\in H$ we
have $C_V(h)=0$, so that $C_G(h)$ is a conjugate of $H$, of order $q-1$, and
$|h^{G}|=q$. For $vh$ with $v\ne0$, $h\ne1$ the situation is the same up to
conjugacy. Hence $n=q$.
\end{proof}

Thus the ratio $\log|G|/\log n$ on the series $\AGL(1,q)$ tends to~$2$, and the
exponent~$2$ in Conjecture~\ref{conj:sharp} cannot be improved. Note that the
extremal groups here are soluble rather than perfect.

\medskip
\noindent\textbf{Computational check.} Using GAP~\cite{GAP} we ran through all
groups of order at most $1000$ with trivial centre. No case with $|G|>n^{2}-n$
was found. The maximum of the quantity $\log|G|/\log n$ is attained exactly on
the groups of Proposition~\ref{prop:agl}:
\[
  \begin{array}{lccc}
    \text{group} & |G| & n & \log|G|/\log n\\[2pt]
    \AGL(1,19) & 342 & 19 & 1.9816\\
    \AGL(1,32) & 992 & 32 & 1.9908
  \end{array}
\]

The following observation shows that the truth of
Conjecture~\ref{conj:sharp} implies the truth of Conjecture~\ref{NVL}.

\begin{corollary}\label{cor:conjC}
Conjecture~\ref{conj:sharp} implies Conjecture~\ref{NVL}.
\end{corollary}

\begin{proof}
Let $G$ be perfect and an $n$-$\bfc$-group, and put $\bar G=G/\Z(G)$. The class
$\bar x^{\bar G}$ is the image of the class $x^{G}$, so $\bar G$ is an
$\bar n$-$\bfc$-group with $\bar n\le n$; by Lemma~\ref{lem:grun} we have
$\Z(\bar G)=1$. Applying Conjecture~\ref{conj:sharp} to $\bar G$, we obtain
$|G:\Z(G)|=|\bar G|\le \bar n^{2}-\bar n\le n^{2}-n<n^{2}$.
\end{proof}

\begin{remark}
Note that Conjecture~\ref{NVL} says nothing about non-perfect groups with
trivial centre. Moreover, the extremal series of Proposition~\ref{prop:agl} is
soluble. Thus a proof of Conjecture~\ref{NVL} would not directly yield
Conjecture~\ref{conj:sharp}.
\end{remark}

\section{Open questions}\label{sec:open}

Naturally, settling Conjecture~\ref{conj:sharp} is the main object of study. A
first step towards a proof of it is the following question.

\begin{question}
Is there a polynomial bound for the order of an $n$-$\bfc$-group with trivial
centre?
\end{question}

Another weakening of Conjecture~\ref{conj:sharp} is to prove it for soluble
groups.

\begin{question}
Is Conjecture~\ref{conj:sharp} true for soluble groups?
\end{question}

Here, instead of $\Aut(G')$, one may work with the Fitting subgroup and apply
the technique of Chapter~III of~\cite{NVL77}, where the bound
$|G'|\le n^{\frac12(5+\log n)}$ is obtained in the soluble case.

It would likewise be an important advance to find a better bound for the number
of generators.

\begin{question}
Is there a linear function $f$ such that $\dd(G)\leq f\bigl(\log \bfc(G)\bigr)$
for every finite group $G$ with trivial centre?
\end{question}

	\bigskip

Ilya~B. Gorshkov

Sobolev Institute of Mathematics,

Novosibirsk, Russia,

E-mail address: ilygor8@gmail.com
\end{document}